\documentclass[12pt]{amsart}
\usepackage[colorlinks=true,citecolor=blue,linkcolor=magenta]{hyperref}
\usepackage{amsmath,mathtools}
\usepackage{verbatim}
\usepackage{amsfonts}
\usepackage{amssymb}
\usepackage{color,colortbl}
\usepackage{enumerate}
\usepackage[top=1in, bottom=1in, left=1in, right=1in]{geometry}
\usepackage{mathrsfs}
\usepackage{blkarray}
                   
\usepackage{stmaryrd}
\usepackage{bbold}

\usepackage{cleveref}
\usepackage{soul}
\usepackage{mathrsfs}
\usepackage{tikz-cd}
\usepackage{tikz}
\usepackage{mathtools}
\theoremstyle{definition}
\newtheorem{theorem}{Theorem}[section]
\newtheorem{theoremx}{Theorem}
\newtheorem*{theorem*}{Theorem}

\numberwithin{equation}{section}

\newtheorem{question}[theorem]{Question}
\newtheorem{corollary}[theorem]{Corollary}
\newtheorem{lemma}[theorem]{Lemma}
\newtheorem{proposition}[theorem]{Proposition}

\newcommand{\tr}{\operatorname{tr}}

\theoremstyle{definition}
\newtheorem{definition}[theorem]{Definition}

\newtheorem{example}[theorem]{Example}

\newtheorem*{conjecture*}{Berger's Conjecture}
\newtheorem{remark}[theorem]{Remark}

\newtheoremstyle{TheoremNum}
{8pt}{8pt}              
{\upshape}                      
{}                              
{\bfseries}                     
{.}                             
{.5em}                             
{\theoremname{#1}\theoremnote{ \bfseries #3}}
\theoremstyle{TheoremNum}

\newcommand{\m}{\mathfrak{m}}
\newcommand{\n}{\mathfrak{n}}

\renewcommand{\(}{\left(}
\renewcommand{\)}{\right)}

\newcommand{\NN}{\mathbb{N}}

\newcommand{\cC}{\mathfrak{C}}

\newcommand{\Rank}{\operatorname{rank}}

\newcommand{\Hom}{\operatorname{Hom}}

\newcommand{\Ext}{\operatorname{Ext}}

\newcommand{\Tor}{\operatorname{Tor}}

\newcommand{\depth}{\operatorname{depth}}

\newcommand{\h}{\operatorname{h}}

\renewcommand{\leq}{\leqslant}
\renewcommand{\geq}{\geqslant}
\newcommand{\ds}{\displaystyle}

\DeclareMathOperator{\im}{im}

\newcommand{\tens}{\otimes}

\title{Partial Trace Ideals, Torsion and Canonical Module}
\author{Sarasij Maitra}
\date{}
\email{sm3vg@virginia.edu}
\address{University of Virginia, Charlottesville, VA}
\begin{document}
	\begin{abstract}
		For any finitely generated module $M$ with non-zero rank over a commutative one dimensional Noetherian local domain, the numerical invariant $\h(M)$ was introduced and studied in \cite{maitra2020partial}. We establish a bound on it which helps capture information about the torsion submodule of $M$ when $M$ has rank one and generalizes the discussion in \cite{maitra2020partial}. 
		We further study bounds and properties of $\h(M)$ in the case when $M$ is the canonical module $\omega_R$. This in turn helps in answering a question of S. Greco and then provide some classifications. Most of the results in this article are based on the results presented in the author's doctoral dissertation \cite{phdthesis}.
	\end{abstract}

\maketitle

\section*{Introduction}\label{torsioninterpret1}

Let $(R,\m,k)$ be a Noetherian local domain and let $M$ be a finitely generated $R$-module. The study of vanishing of the torsion submodule of $M$ has been of high interest among algebraists. There exist quite a few open questions along these lines even when $R$ has dimension one, e.g. the Huneke-Weigand Conjecture \cite{huneke1994tensor}, the Berger Conjecture \cite{MR152546}, etc. 

Denote the torsion submodule of $M$ by $\tau(M)$. If $Q$ denotes the fraction field of $R$, then $$\tau(M):=\ker(M\to M\otimes_R Q).
$$
In order to have discussions regarding $\tau(M)$, a natural question to ask is how do we interpret the torsion beyond the above definition. If $M$ is of the form $M_1\tens_RM_2$, then detailed discussions exist in the literature when the problem of vanishing of torsion is interpreted in terms of \textit{rigidity} of $\Tor$ and $\Ext$ modules (see for example, \cite{auslander}, \cite{lichtenbaum1966vanishing}, \cite{murthy63}, \cite{huneke1994tensor}, \cite{constapel1996vanishing}, \cite{huneke2001vanishing}, \cite{celikbas2011vanishing}, \cite{Celikbas2015VanishingOT}, \cite{huneke2019rigid}, etc.). 
In this article, we interpret the torsion of a rank one module over a one-dimensional local domain, in terms of maps $f:M\to R$. Recall that $\Rank_R(M):=\dim_Q(M\otimes_R Q)$. 

More precisely, suppose $M$ has rank one, then we get the following interpretation of $\tau(M)$. 
Let $J$ be an ideal to which $M$ surjects. Let $f$ denote the surjection. Since both $J$ and $M$ have rank one, we get the following exact sequence:
$$0\to \tau(M)\to M\to J\to 0.$$ Thus, $J\cong \frac{M}{\tau(M)}$ for any ideal $J$ which appears as the image of an element in $\Hom_R(M,R)$. (see \Cref{torsioninterpret} for the proof of the above statement). Thus studying such ideals $J$ is a natural approach to capture some information about $\tau(M)$. This precisely motivated the definition \cite[Definition 2.1]{maitra2020partial} in the case $R$ is a one-dimensional local domain. We recall the definition here.
$$\h(M):=\min\{\lambda(R/J)\mid M\to J\to 0,J\subseteq R\}$$ where $\lambda(\cdot)$ denotes length. We call any such $J$ that achieves the above minimum as a \textit{partial trace ideal of $M$}, or equivalently say that \textit{$J$ realizes $M$}. Recall that the trace ideal of $M$ is defined to be $\tr_R(M):=\sum_{f:M\to R}f(M)$, and hence any such $J$ as above is indeed a partial trace ideal. By a partial trace ideal $I$, we mean that $I$ is a partial trace of itself.

  We should mention here that the above definition is closely related to \cite[Definition 1.1]{greco1984postulation} in S. Greco's work which we found after the work in \cite{maitra2020partial} was published. We also point out that the definition in \cite{greco1984postulation} is restricted only to ideals in a one dimensional local domain, whereas we study it for any finitely generated module. We do recover certain results which were proved in S. Greco's article, but his work was not aimed at dealing with $\tau(M)$.

  The case when $R$ is a one dimensional regular local ring is quite well-known. Since a one dimensional regular local ring is a principal ideal domain, a finitely generated module is torsion free if and only if it is free. So, we only focus on non-regular rings. Moreover, if $M$ is cyclic of rank one, then again $M\cong R$, hence $\tau(M)=0$. This brings us to the first main result of this article (see \Cref{MainThm}). Here $\mu(\cdot)$ denotes the minimal number of generators. 

\begin{theoremx} Let $(S,\n,k)$ be a regular local ring and $\ds R={S}/{I}$ be a non-regular one dimensional domain with maximal ideal $\m$ such that $\mu(\m)=n$. Let $M$ be a rank one $R$-module with $\mu(M)\geq 2$. Assume that $I\subseteq \n^{s+1}$ for some $s\geq 1$ and $J\subseteq \m^s$ for some partial trace ideal of $M$. 
	Then $\tau(M)\neq 0$ if \[ \h(M)<\frac{-\dim_k\Tor^R_1(M,k)+\mu(I)}{n}+{n+s\choose s}\frac{s}{s+1}.\]
\end{theoremx}
\noindent Note that \cite[Theorem A]{maitra2020partial} is an immediate application of the above theorem. 

We study this invariant on ideals using connections with the $I$-Ulrich modules as defined in \cite{dao2021reflexive}. We also study some general bounds on $\h(M)$ for any finitely generated $R$-module $M$ (not necessarily rank one). The key tool is \cite[Theorem 2.10]{maitra2020partial}. This also in turn leads us to a question that was asked informally by S. Greco in a conversation with C. Huneke. S. Greco had asked if $$\min\{\lambda(R/J)\mid J\cong \omega_R \}\leq \lambda({R}/{\cC})+e(R)-\lambda(\overline{R}/R)$$ where $\cC$ denotes the conductor ideal, $e(R)$ is the multiplicity of $R$, $\overline{R}$ is the integral closure of $R$ in $Q$ and $\omega_R$ is the canonical module which can be identified with an ideal of $R$. Note that this question is specifically asking whether $\h(\omega_R)$ satisfies the bound. This leads us to the second main result in this article where we show that the answer is negative (see \Cref{Greco}). Recall that analytically unramified means that the $\m$-adic completion $\widehat{R}$ is reduced. In this case, a canonical module exists \cite[Corollary 1.7]{MR1296598}. 

\begin{theoremx}\label{thmB}
	Let $(R,\m,k)$ be a one dimensional analytically unramified local domain with infinite residue field $k$. Let  the canonical module be $\omega_R$. Then the following statements hold.
	\begin{enumerate}
		\item $\h(\omega_R)=0$ if and only if $R$ is Gorenstein. 
		\item If $R$ is not Gorenstein, then $$e(J;R)-\mu(\omega_R)+1\geq \h(\omega_R)\geq \lambda(R/\cC)+e(R)-\lambda(\overline{R}/R)$$ where $J$ is any ideal that realizes the canonical module.

	\end{enumerate}
\end{theoremx}
We provide an example to show that strict inequalities can occur in the above statement. This example was mentioned to the author by C. Huneke. However, we note that the upper bound is achieved if and only if $R$ is almost Gorenstein (see \Cref{almostgorenteinpropo}). The ring $R$ is said to be almost Gorenstein if $\lambda(\overline{R}/R)=\lambda(R/\cC)+\mu(\omega_R)-1$ \cite{barucci1997one}. Similarly, we discuss the condition when the lower bound is achieved (see \Cref{intclos}). We also discuss that $\h(\omega_{R})\geq 2$ whenever $R$ is not Gorenstein. 

The next result is another classification of a one dimensional local domain in terms of $\h(\omega_{R})$ (see \Cref{ffGcase}). 

\begin{theoremx}
Let $(R,\m,k)$ be a one dimensional analytically unramified non-Gorenstein local domain with infinite residue field $k$ and assume that $\overline{R}$ is a DVR. Let $\omega_R$ denote the canonical module of $R$. Then $\tr_R(\omega_R)=\cC$ if and only if $\h(\omega_R)=2\lambda(R/\cC)$. 	
\end{theoremx}
Such rings were very recently termed as \textit{far-flung Gorenstein} rings in the \textit{numerical semi-group} case.  We refer the reader to \cite{herzog2021tiny} for further details.

The above two results shed some light on the question \cite[Conclusion 4]{maitra2020partial} that was asked in the author's previous work, regarding the classification of one dimensional local domains using $\h(\omega_R)$ as a tool.

\subsection*{Acknowledgements} I am highly grateful to Craig Huneke for detailed discussions regarding the topics covered in this short article and also for informing me about the question of S. Greco as well as \Cref{strict}. I am also deeply indebted to Vivek Mukundan. Most of the results in this article are based on the results presented in \cite{phdthesis}.

\section{Preliminaries}

Throughout this article, $(R,\m,k)$ denotes a local Noetherian one dimensional domain unless otherwise specified. We denote by $Q$, the fraction field of $R$, and let $\overline{R}$ denote the integral closure of $R$ in $Q$. All modules $M$ considered will be finitely generated. We write $\mu(\cdot)$ and $\lambda(\cdot)$ to denote minimal number of generators and length respectively. The rank of $M$ is defined to be $\dim_Q(M\otimes_R Q)$. Let $\widehat{R}$ denote the $\m$-adic completion of $R$.  

By $\cC$, we denote the conductor ideal $R:_Q\overline{R}$. Observe that $\cC\neq 0$ if and only if $\overline{R}$ is a finitely generated module over $R$. This condition is guaranteed when we assume that $\widehat{R}$ is reduced \cite[Corollary 4.6.2]{MR2266432}. We say $R$ is analytically unramified if $\widehat{R}$ is reduced. 
We denote the Hilbert-Samuel multiplicity of a module $M$ with respect to an $\m$-primary ideal $I$ by $e(I;M)$. When $I=\m$, we simply write $e(M)$. An ideal $J\subseteq I$ is called a reduction of $I$ if there exists $r\neq 0$ such that $I^{r+k}=JI^k$ for all $k\geq 1$. It is a minimal reduction if no ideal strictly contained in $J$ is a reduction of $I$. Reductions play an important role in the theory of multiplicity. We refer the interested reader to numerous sources such as  \cite{serre1965algebre}, \cite{serre1997algebre}, \cite{bruns_herzog_1998}, \cite{MR2266432} for further details on multiplicity. 

A module $M$ is called maximal Cohen-Macaulay if $\depth M=\dim R$ where $\depth(M)=\min\{i\mid \Ext^i_R(k,M)\neq 0\}$. An MCM $M$ is called $I$-Ulrich for some ideal $I$ if $e(I;M)=\lambda(M/IM)$. We refer the reader to \cite{dao2021reflexive} for further details regarding $I$-Ulrich modules. The most important property of $I$-Ulrich modules that we will use is $IM=xM$ for some minimal reduction $x$ of $M$, or equivalently, $IM\cong M$ \cite[Proposition 4.5]{dao2021reflexive}.

  We denote a canonical module of $R$ by $\omega_R$. Such a module exists if and only if $R$ is the quotient of a Gorenstein ring. Since $R$ is a domain in our case,  $\omega_{R}$ can be further identified with an ideal of $R$ \cite[Theorem 3.3.6, Proposition 3.3.18]{bruns_herzog_1998}.

\begin{definition}\label{definv}\cite[Definition 1.1]{maitra2020partial} Let $R$ be a local Noetherian one dimensional domain. For any $R$-module $M$, define
	$$\h(M):=\min\{\lambda(R/J)~|~M\to J\to 0, J\subseteq R\}$$ where $\lambda(\cdot)$ denotes the length as an $R$-module. We say that an ideal \textbf{\textit{$\ds J$ realizes $M$}}  if $M$ surjects to $J$ and $\ds \h(M)=\lambda(R/J)$. Equivalently, we say $J$ is a partial trace ideal of $M$. If no $M$ is specified, then a partial trace ideal $J$ means that it is a partial trace ideal of itself.
\end{definition}

Also, we recall that ideals $I$ and $J$ are isomorphic means that $I=\alpha J$ for some $\alpha\in Q$ \cite[Remark 2.2]{maitra2020partial}. 


As we proceed through the paper, in every section we will set up additional notations and conventions whenever necessary and also make all the hypothesis explicit for convenience.


\section{Torsion}
We discussed in the introduction that the torsion submodule for a finitely generated $R$-module $M$ is given by $$\tau(M)=\ker(M\to M\otimes_R Q)$$ where $Q$ is the fraction field of $R$. We prove that when $M$ has rank one, then $\tau(M)$ can be interpreted via the exact sequence mentioned in the introduction. This is quite well-known.

\begin{proposition}\label{torsioninterpret}
	Let $R$ be any Noetherian local domain with fraction field $Q$. Let $M$ be a finitely generated module of rank one with torsion submodule $\tau(M)$. Then for any $f\in \Hom_R(M,R)$ such that $f\neq0$, we have $$0\to \tau(M)\to M\to \im(f)\to 0$$ where $\im(f)$ denotes the image of $f$. Thus, $\im(f)\cong M/\tau(M)$ for any $0\neq f\in \Hom_R(M,R)$.
\end{proposition}

\begin{proof}
	Let $J=\im(f)$. Then $J$ is an ideal of $R$. Consider the following diagram. 
	\[\begin{tikzcd}[ampersand replacement=\&]
	0\ar[r]\& \ker f\ar[r] \& M\ar[r,"f"]\ar[d,"i_M"] \& J\ar[r]\ar[d,"i_J"]\& 0\\
	~\& ~ \&  {M\tens_R Q} \ar[r,"f_Q"]\& J\tens_R Q \& ~ 
	\end{tikzcd}\] 
	The vertical arrows are the natural maps, i.e., $\tau(M)=\ker (i_M)$ and $\tau(J)=\ker(i_J)$. 
	
	Clearly, $\tau(M)\subseteq \ker f$. Since $\Rank(M)=\Rank(J)=1$,  $f_Q$ is an isomorphism. Since $\tau(J)=0$, $i_J$ is injective. This shows that $\tau(M)=\ker f$ by a diagram chase.
\end{proof}

\begin{remark}
	Note that if $(R,\m,k)$ is a regular local ring of dimension one, then $\tau(M)=0$ if and only if $M\cong R$ for any finitely generated rank one $R$-module $M$. First notice that $R$ is a DVR and hence a PID \cite[Proposition 9.2]{atiyah}. Now we can apply the structure theorem for modules over PIDs \cite[Theorem 5, Chapter 12]{dummit}. Thus, in order to study $\tau(M)$ for rank one modules, we henceforth assume that $R$ is non-regular, i.e., $\mu(\m)\geq 2$.
\end{remark}
\begin{proposition}\label{h=0}
	The following statements hold for any finitely generated module $M$ of rank one over a one-dimensional local domain $R$ with $\mu(\m)\geq 2$.
	\begin{enumerate}
		\item If $\mu(M)=1$, then $M\cong R$. Thus, $\h(M)=0$.
		
		\item If $\mu(M)\geq 2$ and $\h(M)=0$, then $\tau(M)\neq 0$.  
	\end{enumerate}
\end{proposition}
\begin{proof}
Since $M$ has a rank, $\tau(M)\neq M$ and so there exists a non-zero $f:M\to R$. Since $M=Rx$, any $y\in \tau(M)$ is of the form $y=rx$ and is in $\ker f$. So $rf(x)=f(rx)=f(y)=0$ but this implies that $f(x)=0$, a contradiction since $f\neq 0$. Thus, $\tau(M)=0$. Notice that $f(M)$ is the principal ideal $f(x)R$ and any principal ideal is isomorphic to $R$. Thus composing $f$ with this isomorphism gives the desired conclusions of $(1)$.

Since $\h(M)=0$, there exists an exact sequence $0\to \tau(M)\to M\to R\to 0$ and hence it splits. So, $M\cong R\oplus \tau(M)$. Since $\mu(M)\geq 2$, we have proved $(2)$.	
\end{proof}
\begin{remark}
	So to study torsion of a rank one module which is generated by at least two elements, it is naturally more interesting to look at the case $\h(M)\geq 1$. Thus, for the rest of this section, we will assume that $\h(\cdot)$ is a positive integer. Note that this necessarily implies that any surjective image of such an $M$ is not a principal ideal: to see this, note that a principal ideal is isomorphic to $R$ and thus we will have a surjection of $M$ to $R$ implying $\h(M)=0$, a contradiction to the assumption.
\end{remark} 
	
	
	We shall explore further properties of $\h(\cdot)$ independently in the subsequent sections without having this positivity assumption.
	
	\begin{remark}\label{maximalidealgen}\cite[Remark 4.3]{maitra2020partial}
		Suppose $\ds (S,\n,k)$ is a regular local ring of embedding dimension $n$ and $\ds R=S/I$ for some ideal $I$ in $S$ where $\ds I\subseteq \n^{s+1}$ for some $s\geq 1$. Letting $\m$ denote the maximal ideal of $R$, we have $\mu(\n^i)=\mu(\m^i)={n+i-1\choose i}$ for $1\leq i\leq s$ due to the condition imposed on $I$.
	\end{remark}

\begin{lemma}\cite[Proposition 4.5]{maitra2020partial}\label{lem1}
	Let $\ds (S,\n,k)$ be a regular local ring of embedding dimension $n$ and let $\ds R=S/I$ for an ideal $I$ in $S$. Let $\m$ be the maximal ideal of $R$. Further assume that $\ds I\subseteq \n^{s+1}$ for some $s\geq 1$. Then $$\ds \operatorname{dim}_k\(\Tor^R_1(\m^s,k)\)=s{n+s-1\choose s+1}+\mu(I).$$
\end{lemma}

\begin{lemma}\label{lengthlemma}\cite[Lemma 4.6]{maitra2020partial}. 
	Let $(R,\m,k)$ be a Noetherian local ring and $M$ be an $R$-module of finite length. Then $$\ds \text{dim}_k\Tor^R_1 (M, k) \leq \lambda(M )(\dim_k\Tor^R_1(k,k)-1)+\mu(M)\leq \lambda(M)\dim_k\Tor^R_1(k,k).$$  
\end{lemma}

The following theorem generalizes \cite[Theorem 4.7]{maitra2020partial}. For convenience of writing, we use $\beta_1(M)$ in the proof to denote $\dim_k\Tor^R_1(M,k)$.

\begin{theorem}\label{MainThm}
	Let $(S,\n,k)$ be a regular local ring and $\ds R={S}/{I}$ be a one dimensional domain with maximal ideal $\m$ and embedding dimension $n$. Let $M$ be a rank one $R$-module with $\mu(M)\geq 2$. Assume that $\ds I\subseteq \n^{s+1}$ for some $s\geq 1$ and $J\subseteq \m^s$ for some ideal $J$ that realizes $M$. If $$\h(M)<\frac{-\dim_k\Tor^R_1(M,k)+\mu(I)}{n}+{n+s\choose s}\(\frac{s}{s+1}\),$$  then $\tau(M)\neq 0$.
\end{theorem}

\begin{proof} We have that $\h(M)=\lambda(R/J)$ since $J$ realizes $M$. 
	Using additivity of length and the assumption $J\subseteq \m^s$, we have $$\ds \lambda\({\m^s}/{J}\)=\lambda(R/J) -1-\sum_{i=1}^{s-1}\mu(\m^i).$$ 
	Using \Cref{lengthlemma},  we get $\ds\dim_k\Tor^R_1\(\frac{\m^s}{J},k\)\leq \lambda\({\m^s}/{J}\)\dim_{k}\Tor^R_1(k,k)$. Thus we get,
	\begin{equation}\tag{\ref{MainThm}.1}\label{eqq2}
	\dim_k\Tor^R_1\(\frac{\m^s}{J},k\)\leq \mu(\m)\(\lambda(R/J) -1-\sum_{i=1}^{s-1}\mu(\m^i)\)
	\end{equation} 
	
	\noindent Tensoring the exact sequence $ 0\to J\to \m^s\to {\m^s}/{J}\to 0$  with $\ds k$, we get {
		$$ \Tor^R_1(J,k)\to \Tor^R_1(\m^s,k)\to \Tor^R_1\(\frac{\m^s}{J},k\)\to J\tens k\to \m^s\tens k \to \frac{\m^s}{J}\tens k \to 0$$} as part of a long exact sequence. Hence,
	\begin{equation}\tag{\ref{MainThm}.2}\label{eqq3}
	\text{dim}_k\Tor^R_1\(\frac{\m^s}{J},k\)\geq -\text{dim}_k\Tor^R_1(J,k)+\text{dim}_k\Tor^R_1(\m^s,k) +\mu(J)-\mu(\m^s)+\mu\({\m^s}/{J}\)
	\end{equation}
	Combining \Cref{eqq2,eqq3}, we get 
	{ \[{\lambda(R/J)\geq \frac{-\text{dim}_k\Tor^R_1(J,k)+\text{dim}_k\Tor^R_1(\m^s,k)  +\mu(J)-\mu(\m^s)+\mu\(\frac{\m^s}{J}\)}{\mu(\m)}+1+\sum_{i=1}^{s-1}\mu(\m^i) }.\]}
	
	Recall that we have the exact sequence $0\to \tau(M)\to M\to J\to 0.$
	Suppose, on the contrary, that $\tau(M)=0$, then $J\cong M$ and hence 
	$\dim_k\Tor^R_1(J,k)=\beta_1(M):=\dim_k\Tor^R_1(M,k)$ and 
	$\mu(J)=\mu(M)$. Further from \Cref{maximalidealgen}, we have that  $\mu(\m^i)={n+i-1\choose i}\text{~~for~~} 1\leq i\leq s.$ So, $ 1+\sum_{i=1}^{s-1}\mu(\m^i)=\sum_{i=0}^{s-1}{n-1+i \choose i}={n+s-1\choose s-1}$.  Combining these observations along with \Cref{lem1}, we get that 
	{ \begin{align*}\lambda(R/J)&\geq \frac{-\beta_1(M)+\mu(I)+s{n+s-1\choose s+1} +\mu(M)-\mu(\m^s)+\mu\(\frac{\m^s}{J}\)}{n}+{n+s-1\choose s-1}\\
		&=\frac{-\beta_1(M)+\mu(I)+\mu(M)-\mu(\m^s)+\mu\(\frac{\m^s}{J}\)}{n}+T
		\end{align*}}
	where { \begin{align*}
		T&=\frac{s{n+s-1\choose s+1}}{n}+{n+s-1\choose s-1}=\frac{s(n+s-1)!}{n(s+1)!(n-2)!}+\frac{(n+s-1)!}{(s-1)!n!}\\
		&=\frac{(n+s-1)!}{(s-1)!n(n-2)!}\(\frac{1}{s+1}+\frac{1}{n-1}\)
		=\frac{(n+s-1)!}{(s-1)!n!}\(\frac{n+s}{s+1}\)\\
		&=\frac{(n+s)!}{s!n!}\(\frac{s}{s+1}\)={n+s\choose s}\(\frac{s}{s+1}\).
		\end{align*}}
	So we conclude that \begin{equation}\tag{\ref{MainThm}.3}\label{eqq4}\lambda(R/J)\geq \frac{-\beta_1(M)+\mu(I)+\mu(M)-\mu(\m^s)+\mu\(\frac{\m^s}{J}\)}{n}+{n+s\choose s}\(\frac{s}{s+1}\)	.\end{equation}
	Since $\mu(\m^s/J)\geq \mu(\m^s)-\mu(J)=\mu(\m^s)-\mu(M)$, we get that $$\lambda(R/J)\geq \frac{-\beta_1(M)+\mu(I)}{n}+{n+s\choose s}\(\frac{s}{s+1}\).$$  But this  a contradiction to the hypothesis on $\h(M)$. Hence, $\tau(M)\neq 0$.
\end{proof}	
%
%

\begin{remark}\label{mainrmk12} We jot down a few observations in connection to the proof of \Cref{MainThm}.
	\begin{itemize}
		\item[a)] As in the proof of \Cref{MainThm}, we can tensor $k$ with the short exact sequence $0\to \tau(M)\to M\to J\to 0$,  and then truncate the long exact sequence to obtain that $\ds \mu(\tau(M))\geq \mu(M)-\mu(J)$. So, if  $\ds M$ surjects to an ideal $J$ which is generated by less than $\mu(M)$ elements, then $\tau(M)\neq 0$.

		\item[b)] 
		We can repeat the proof of \Cref{MainThm} up to \Cref{eqq4} and plug in $\mu(\m^s)=\binom{n+s-1}{s}$ directly to get
		{\small\begin{align*}
			\lambda(R/J) &\geq \frac{-\beta_1(M)+\mu(I)+\mu(M)+\mu\({\m^s}/{J}\)}{n} +{n+s\choose s}\(
			\frac{s}{s+1}\)-\frac{{n+s-1\choose s}}{n}\\
			&=\frac{-\beta_1(M)+\mu(I)+\mu(M)+\mu\({\m^s}/{J}\)}{n}+{n+s-1\choose s-1}\frac{s^2+s(n-1)-1}{s(s+1)}.
			\end{align*}}
		Hence, $\tau(M)\neq 0$ if {\small $$\h(M)<\frac{-\beta_1(M)+\mu(I)+\mu(M)+\mu\({\m^s}/{J}\)}{n}+{n+s-1\choose s-1}\frac{s^2+s(n-1)-1}{s(s+1)}.$$}
	\end{itemize}
\end{remark}

\section{Some Computations of $\h(\cdot)$}

Here we discuss the computation of $\h(\cdot)$ in certain cases. If $(R,\m,k)$ is regular of dimension one, then all ideals are principal and there is nothing interesting to prove. So we assume that $\mu(\m)\geq 2$ throughout the rest of the paper. 

We let $\cC=R:_Q\overline{R}$ be the conductor ideal of $R$. This is the largest ideal shared by $R$ and $\overline{R}$. We may also assume that $R$ is analytically unramified, i.e., $\widehat{R}$ is reduced. This ensures that $\cC\neq 0$ and $\overline{R}$ is a finitely generated $R$-module.

\begin{remark}\label{rem3.1}
	For any $R$-module $M$, if $J$ realizes $M$ for some ideal $J$, then $\ds \lambda(R/J)=\h(J)=\h(M).$ In particular, $J$ realizes itself.
\end{remark}

\begin{proof}
	Suppose $\ds \h(J)=\lambda(R/I)$ for some  $I$ such that there exists a surjection $\phi:J\to I$ (in fact, an isomorphism since $J$ is torsion-free) with $\ds \lambda(R/I)<\lambda(R/J)$. Also let $\ds M\xrightarrow{f} J\to 0$ be the surjection to $J$. Then $M\xrightarrow{\ds \phi \circ f} I\to 0$ and this implies that $\h(M)=\lambda(R/I)<\lambda(R/J)$, a contradiction to the hypothesis that $J$ realizes $M$.
\end{proof}

Before moving on, we should point out that isomorphic ideals can have different lengths as the following example shows. 

\begin{example}
	$ R={k[[x,y]]}/{(x^2+y^3)}$. Note that $I_1=(x,y)R$ and $\ds I_2=(x^2,xy)R$ are isomorphic since, $xI_1=I_2.$ However, $\lambda\({R}/{I_2}\)=\lambda\(\frac{k[[x,y]]}{(x^2,xy,y^3)}\)=4$ whereas $\lambda\({R}/{I_1}\)=1$.
\end{example} 

\begin{theorem}\cite[Theorem 2.5, Theorem 2.10 ]{maitra2020partial}\label{mainpropo}
	Let $\ds (R,\m,k)$ be a one dimensional analytically unramified  non-regular local domain with integral closure $\overline{R}$ and fraction field $Q$. Further assume that  $\omega_R$ exists. For any ideal $J$ of $R$, consider the following statements.
	\begin{itemize}
		\item[$(a)$] $\ds \h(J)=\lambda\({R}/{J}\)$.
		
		\item[$(b)$] $\ds R:_QJ\subseteq \overline{R}$.
		
		\item[$(c)$] $\cC\omega_R\subseteq J\omega_R$.
		\end{itemize}  
	Then $(c)\iff (b)\implies (a)$. Moreover, if $\overline{R}$ is a DVR, then all the statements are equivalent. 
\end{theorem}

\Cref{mainpropo} is going to be our key tool for the rest of this article. For example, it shows that any ideal containing $\cC$ is a partial trace ideal of itself \cite[Corollary 2.6]{maitra2020partial}. This recovers \cite[Lemma 3.1]{greco1984postulation}, namely, that $\h(\cC)=\lambda(R/\cC)$. This also shows that any trace ideal realizes itself \cite[Proposition 3.4]{maitra2020partial}. Further, we should mention here that part $(c)$ does not depend on the choice of the canonical module $\omega_R$, as long as we choose it to be an $R$-submodule of $Q$. In this case, we can use the identifications of $\Hom$ with colons.

The case when the conductor $\cC$ becomes equal to some power of the maximal ideal $\m$ is often interesting. For instance, extensive discussions of such cases exist in \cite{orecchia1981points}. In such cases, we can often describe the invariant $\h(\cdot)$ for high enough powers of $\m$. 

\begin{proposition}\label{IUlrich}
	Let $\overline{R}$ be a finitely generated module over a one dimensional non-regular local ring $R$ and let $\cC$ be the conductor. Then for any ideal $I$, we have $$\h(I\overline{R})=\h(I\cC)=\lambda(R/\cC).$$ 
\end{proposition}

\begin{proof}
	By \cite[Proposition 4.6, Corollary 4.10]{dao2021reflexive}, we get that $\overline{R}$ and $\cC$ are $I$-Ulrich for all ideals $I$, i.e., $I\overline{R}\cong \overline{R}, I\cC\cong \cC.$ Applying $\h(\cdot)$, we get that $\h(I\overline{R})=\h(I\cC)=\h(\cC)$. The proof is now complete by \Cref{mainpropo}. 
\end{proof}

\begin{corollary}\label{condpower}
	Let $\ds (R,\m,k)$ be a one dimensional non-regular analytically unramified domain. If $\cC\cong I^N$ for some $N$, then $\h(I^\ell)=\lambda(R/\cC)$ for all $\ell\geq N$.
\end{corollary} 

\begin{proof}
	Note that $I^k\cC\cong I^{N+k}$ for all $k\geq 0$. The proof is now immediate from \Cref{IUlrich}. 
\end{proof}

\begin{corollary}\label{greconfold}
	Let $\ds (R,\m,k)$ be a one dimensional analytically unramified non-regular local domain. If $\cC=\m^N$ for some $N$ and $\mu(\m^{i})={n+i-1\choose i}$ for $1\leq i\leq N-1$ where $n=\mu(\m)$, then $\ds \h(\m^\ell)={N+n-2\choose n}$ for all $\ell\geq N$.
\end{corollary}

\begin{proof}
	By \Cref{condpower}, $\h(\m^\ell)=\lambda(R/\cC)$ for all $\ell\geq N$. Note that $$\lambda(R/\m^N)=1+\sum_{i=1}^{N-1}\mu(\m^{i})=\sum_{i=0}^{N-1}{n+i-1\choose i}={N+n-1\choose n}.\qedhere$$
\end{proof}
Note that \Cref{greconfold} recovers \cite[Lemma 5.5]{greco1984postulation} by putting $n=2$. More generally, the following holds.

\begin{proposition}
	Let $M$ be an $I$-Ulrich module. Then $\h(M)=\h(I^n M)$ for all $n$. In particular, the non-decreasing sequence $\{\h(I^n)\}_{n\in \NN}$ stabilizes for all ideals $I$.  
\end{proposition}

\begin{proof}
	Since $M$ is $I$-Ulrich, we have $I^nM\cong M$ for all $n$ by \cite[Theorem 4.6]{dao2021reflexive}. So, $\h(M)=\h(I^nM)$ for all $n$. For the last statement, first notice that since $I^{l}\subseteq I^{k}$ for $k\leq l$, we have $\lambda(R/I^l)\geq \lambda(R/I^k)$ and hence $\h(I^l)\geq \h(I^k)$ for $k\le l$. Moreover, for $n>>0$, $I^n$ is $I$-Ulrich (\cite[Example 4.2]{dao2021reflexive}), and hence we can apply the first part to finish the proof. 
\end{proof}

Clearly, partial trace ideals of a module are all isomorphic to each other. The following shows that their integral closures always match. We need the fact that for any ideal $I$, $\tr_R(I)=(R:_QI)I$ where $Q$ is the field of fractions of $R$ \cite[Proposition 2.4]{kobayashi2019rings}. 

\begin{proposition}\label{intclosurepartial}
	Let $(R,\m,k)$ be a one dimensional local analytically unramified domain. Assume that $\overline{R}$ is a DVR. Then for any module $M$, any two partial trace ideals have the same integral closure. 
\end{proposition}

\begin{proof}
	Suppose $\ds I\cong J$ such that $I,J$ are partial trace ideals of $M$. Then $R:_Q J\subseteq \overline{R}$ by \Cref{mainpropo}. Now note that $I\subseteq \tr_R(I)=\tr_R(J)=(R:_Q J)J\subseteq \overline{R}J$. Thus, $I\subseteq \tr_R(I)\subseteq J\overline{R}\cap R=\overline{J}$. Since $I$ is also a partial trace ideal, by symmetry we get $\overline{I}=\overline{J}$. 
	\end{proof}

\begin{corollary}
	Let $(R,\m,k)$ be a Noetherian local analytically unramified one dimensional domain such that $\overline{R}$ is a DVR. Then any integrally closed partial trace ideal is a trace ideal. 
\end{corollary}

\begin{proof}
	Notice that $I\subseteq \tr_R(I)=(R:_Q I)I\subseteq I\overline{R}$. Hence, $I\subseteq \tr_R(I)\subseteq I\overline{R}\cap R=\overline{I}$. This finishes the proof.
\end{proof}

\section{Bounds on $\h(\omega_R)$ and some classifications}

This section is aimed at establishing some general bounds on $\h(M)$ for any non-zero rank, finitely generated module $M$ over a one dimensional Noetherian local domain $(R,\m,k)$. Notice that $\h(M)\geq \lambda(R/\tr_R(M))$ clearly. So, we mainly focus on upper bounds. 

Throughout the section we assume that $R$ admits a canonical module. Hence, it can be identified with an ideal \cite[Proposition 3.3.18]{bruns_herzog_1998}. For instance, if we assume that $R$ is analytically unramified, then $R$ admits a canonical module \cite[Corollary 1.7]{MR1296598}.

The following lemma is well-known and appears as parts of proofs in many sources like \cite[Theorem 3]{MR1184039}, \cite[Proposition 2.1]{MR1296598} and \cite[Lemma 2.2]{herzog2022upper}. 
\begin{lemma}\label{prop.lem3}
	Let $R$ be a one dimensional local domain with canonical module $\omega_R$, conductor ideal $\cC$ and integral closure $\overline{R}$. Then $\lambda(\omega_R/\cC \omega_R)=\lambda(\overline{R}/R)$. 
\end{lemma}

\begin{proposition}\label{prop.cor2} 
	Let $\ds (R,\m,k)$ be a one dimensional analytically unramified local domain with integral closure $\overline{R}$ and fraction field $K$. Further assume that  $\overline{R}$ is a $DVR$. Identifying the canonical module with some ideal $\omega_R$ of $R$, we have $$\ds \h(M)\leq \lambda(R/\cC)+\lambda\(\frac{\cC}{\cC\omega_R}\)=\lambda\(\frac{R}{\cC\omega_R}\)$$ for any $R$-module $M$ which has non-zero rank.
\end{proposition}

\begin{proof}
	The canonical module exists and can be identified with an ideal $\omega_R$ of $R$ \cite[Corollary 1.7]{MR1296598}. Let $J$ realize $\ds M$. Then by \Cref{rem3.1}, $J$ realizes itself. So, by \Cref{mainpropo}, we have $\ds \cC\omega_R\subseteq J\omega_R$. Now combining all this data we get,
	\begin{align*}
	\h(M)&=\lambda\({R}/{J}\)
	=\lambda\({R}/{J\omega_R}\)-\lambda\({J}/{J\omega_R}\)
	=\lambda\({R}/{\cC\omega_R}\)-\lambda\({J\omega_R}/{\cC\omega_R}\)-\lambda\({J}/{J\omega_R}\)\\
	&=\lambda\({R}/{\cC}\)+\lambda\({\cC}/{\cC\omega_R}\)-\lambda\({J}/{\cC\omega_R}\)
	\end{align*}
	This finishes the proof.
\end{proof}
\begin{corollary}\label{prop.thm.cor3}
	Let $\ds (R,\m,k)$ be an analytically unramified one dimensional Gorenstein local domain with integral closure $\overline{R}$ and fraction field $K$. Further assume that $\overline{R}$ is a $DVR$. Then $$\h(M)\leq \lambda\(\frac{R}{\cC}\)$$ for any $R$-module $M$ which has non-zero rank.
\end{corollary}

\begin{proof}
	Since $R$ is Gorenstein, $\omega_R$ exists and can be chosen to be $R$ itself. The proof now follows immediately from \Cref{prop.cor2}.
\end{proof}

We can also link $\ds \h(\cdot)$ with another invariant of the ring as the following proposition shows.

\begin{proposition}\label{prop.cor3}
	Let $\ds (R,\m,k)$ be a one dimensional analytically unramified local domain with integral closure $\overline{R}$ and fraction field $K$. Further assume that  $\overline{R}$ is a $DVR$. Then $$\ds \h(M)\leq \h(\omega_R)+\lambda\(\frac{\overline{R}}{R}\)$$ for any $R$-module $M$ that has non-zero rank.
\end{proposition}

\begin{proof}
	Let $J$ realize $M$ and identify $\omega_R$ with some ideal of $R$. We proceed exactly as in \Cref{prop.cor2}.
	\begin{align*}
	\h(M)&=\lambda\(R/J\)
	=\lambda\(R/J\omega_R\)-\lambda\(J/J\omega_R\)
	=\lambda\(R/\cC\omega_R\)-\lambda\(J\omega_R/\cC\omega_R\)-\lambda\(J/J\omega_R\)\\
	&=\lambda\(R/\omega_R\)+\lambda\(\overline{R}/R\)-\lambda\(J/\cC\omega_R\)
	\end{align*} where in the last line we used \Cref{prop.lem3}. Since this is satisfied by all possible identifications of $\omega_R$ inside $R$, 
	the proof is complete by \Cref{definv}.
\end{proof}

\subsection{The Invariant $\h(\omega_R)$}\label{GrecoSec}

Notice that $\ds \h(\omega_R)=0$ if and only if $R$ is Gorenstein. Thus it is natural to ask if $\ds \h(\omega_R)$ can be used as a tool in classification problems. 

The following question was asked by S. Greco in a conversation with C. Huneke.
\begin{question}[S. Greco]
	Suppose $(R,\m,k)$ is a complete one dimensional local domain with canonical module $\omega_R$. Is the following assertion true?
	$$\min\{\lambda(R/J)\mid J\cong \omega_R, J\subseteq R\}\leq \lambda(R/\cC)+e(R)-\lambda(\overline{R}/R).$$
\end{question}
Note that it is asking to check the bound on $\h(\omega_R)$ where $\omega_R$ is a canonical module. We show that the statement is false and provide the correct inequality. 


\begin{theorem}\label{Greco}
	Let $(R,\m,k)$ be a one dimensional analytically unramified local domain with infinite residue field $k$, and canonical module $\omega_R$. Then the following statements hold.
	\begin{enumerate}
		\item $\h(\omega_R)=0$ if and only if $R$ is Gorenstein. 
		\item If $R$ is not Gorenstein, then $$e(J;R)-\mu(\omega_R)+1\geq \h(\omega_R)\geq \lambda(R/\cC)+e(R)-\lambda(\overline{R}/R)$$ where $J$ is any ideal that realizes the canonical module.
	\end{enumerate}
\end{theorem}

\begin{proof}
	Statement $(1)$ is clear. Let $\omega_R$ denote a partial trace ideal of the canonical module. It is a non-principal ideal since $R$ is not Gorenstein. So, we get $$\h(\omega_R)=\lambda(R/\omega_R)=\lambda(R/\cC)+\lambda(\cC/\cC\omega_R)-\lambda(\omega_R/\cC\omega_R).$$ Since $\cC$ is $\omega_R$-Ulrich, by \cite[Proposition 4.5]{dao2021reflexive}, we have $\cC\omega_R=z\cC$ where $z$ is a minimal reduction of $\omega_R$ (exists by \cite[Proposition 8.3.7, Corollary 8.3.9]{MR2266432}).  Thus using \Cref{prop.lem3}, we obtain that
	\begin{equation}\tag{\ref{Greco}.1}\label{primaryeq}\h(\omega_R)=\lambda(R/\cC)+\lambda(\cC/z\cC)-\lambda(\overline{R}/R).\end{equation}
	Since $\cC$ is MCM of rank 1, \cite[Corollary 4.7.11]{bruns_herzog_1998} gives us
	$$\h(\omega_R)=\lambda(R/\cC)+e(z;R)-\lambda(\overline{R}/R).$$ For the upper bound, first notice that $e(z;R)=e(\omega_R;R)$ \cite[Proposition 11.2.1]{MR2266432}. Then using \cite[Proposition 12.2.3]{MR2266432}, we get $$\lambda(R/\cC)-\lambda(\overline{R}/R)\leq 1-\mu(\omega_R).$$ Combining these data, \Cref{primaryeq} immediately yields the required upper bound. 
	
	Further, we have $e(z;R)\geq e(R)$. So, \Cref{primaryeq} immediately gives the lower bound $$\h(\omega_{R})\geq \lambda(R/\cC)+e(R)-\lambda(\overline{R}/R). \qedhere $$
\end{proof}

 Moreover, the following example, which is due to C. Huneke, illustrates that the inequalities in \Cref{Greco} can be strict. It also shows that $\cC\omega_R\subseteq J\omega_R$ does not imply $\cC\subseteq J$ in \Cref{mainpropo}. For more details regarding the computations here, we refer the reader to \cite[Section 5.1]{maitra2020partial}.

\begin{example}[Huneke]\label{strict}
	Let $\ds R=k[[t^5,t^6,t^8]]\cong \frac{k[[x,y,z]]}{(y^3-x^2z,x^2y-z^2,x^4-y^2z)}$ where $k$ is algebraically closed of characteristic $0$. Then $e(R)=5, \overline{R}=k[[t]]$ and $\cC=t^{10}\overline{R}=(x^2,xy,y^2,xz,yz)$. So, $\lambda(R/\cC)=4$ (notice that the missing valuations from $\cC$ which are present in $R$ are $\{0,5,6,8\}$). We also have that $\lambda(\overline{R}/R)=6$ (the valuations that are missing from $R$ but are present in $\overline{R}$ are $\{1,2,3,4,7,9\}$. Thus, $\lambda(R/\cC)+e(R)-\lambda(\overline{R}/R)=4+5-6=3.$
	
	Moreover, we can choose the canonical ideal to be $\omega_R=(t^6,t^8)=(y,z)$. With this choice, notice that $\cC\omega_R\subseteq\omega_R^2$ and hence by \Cref{mainpropo}, we get that $\h(\omega_R)=\lambda(R/\omega_R)=\lambda\(k[[x,y,z]]/(y,z,x^4)\)=4$. Thus, $\h(\omega_R)>\lambda(R/\cC)+e(R)-\lambda(\overline{R}/R)$ in this case. 
	
	Next, notice that $\mu(\omega_R)=2, e(\omega_R;R)=6$. So, $$e(\omega_R;R)-\mu(\omega_R)+1=6-2+1=5>4=\h(\omega_R).$$ 
	
	Finally, notice that $\cC\not\subseteq \omega_R$ even though $\cC\omega_R\subseteq \omega_R^2$. This shows that without the Gorenstein assumption, in general we may not have $\cC$ contained inside a realizing ideal of a module $M$. 
\end{example}

In their work \cite{barucci1997one}, Barucci and Fr\"oberg defined an \textit{almost Gorenstein} ring to be a ring where $\lambda(\overline{R}/R)=\lambda(R/\cC)+\mu(\omega_R)-1$. These classes of rings have been of interest lately: the reader can refer to \cite{barucci1997one}, \cite{goto2015almost}, \cite{herzog2019trace}, \cite{dao2020trace} amongst other sources. This notion immediately yields the following corollary. 
%
%
%
%
%
%
%
%

\begin{corollary}\label{almostgorenteinpropo}
	Let $(R,\m,k)$ be a one dimensional analytically unramified local domain with infinite residue field. Then $R$ is almost Gorenstein (but not Gorenstein) if and only if $$\h(\omega_R)=e(J;R)-\mu(\omega_R)+1$$ for any partial trace ideal $J$ of $\omega_R$. 
\end{corollary}

\begin{proof}Let $J$ be a partial trace ideal of the canonical module. By  \Cref{Greco}, we have $\h(\omega_R)\leq e(J;R)-\mu(\omega_R)+1$. The same proof also shows that equality is achieved if and only if $\lambda(\overline{R}/R)-\lambda(R/\cC)=\mu(\omega_R)-1$ if and only if $R$ is almost Gorenstein.
\end{proof}

%

\begin{proposition}\label{intclos} Let $(R,\m,k)$ be a one dimensional analytically unramified non-Gorenstein local domain with infinite residue field $k$ and assume that $\overline{R}$ is a DVR. Then the integral closure of some canonical ideal (i.e., an isomorphic copy of the canonical module in $R$) is the maximal ideal $\m$ if and only if $$\h(\omega_R)=\lambda(R/\cC)+e(R)-\lambda(\overline{R}/{R}).$$ 
\end{proposition} 

\begin{proof}
	Let $\omega_R$ denote an ideal that realizes the canonical module, and has $z$ as a minimal reduction. Again, from the proof of \Cref{Greco}, we get that $\h(\omega_R)=e(z;R)+\lambda(R/\cC)-\lambda(\overline{R}/R)$. Thus the equality as in the statement holds if and only if $e(z;R)=e(R)$. Since $R$ is formally unmixed, the equality holds if and only if $\overline{\omega_R}=\m$ by \cite[Theorem 11.3.1]{MR2266432}.
	
	Now if we start with any canonical ideal $\omega_R'$, then $\omega_R'\cong \omega_R$. So, ${\omega_R'}\subseteq \tr_R(\omega_R')=\tr_R(\omega_R) \subseteq\overline{\omega_R}$ as in the proof of \Cref{intclosurepartial}. If $\overline{\omega_R'}=\m$, then the same holds for $\omega_R$ as well and we can repeat the first part of the proof.
\end{proof}

\begin{remark}
	In an unpublished work, J. Sally had shown that $\omega_R\not \cong \m$ in a one dimensional non-regular local domain, where $\omega_R$ exists. Indeed if such an isomorphism exists, then  $\omega_R$ is reflexive  and hence $R$ is Gorenstein \cite[Corollary 3.2, Theorem 5.5]{dao2021reflexive}. Thus, $R\cong \m$ which contradicts the non-regularity assumption. 
	

	Thus, $\omega_R\cong \m$ implies that $\h(\omega_R)=0$. So, we conclude that $\h(\omega_R)\geq 2$ if $R$ is non-Gorenstein. Moreover, since $\mu(\omega_{R})\geq 2$, we get $\h(\omega_R)\le e(J;R)-1$ for any partial trace ideal of $\omega_{R}$. 
\end{remark}

In \cite{herzog2021tiny}, J. Herzog, S. Kumashiro and D. Stamate introduced the notion of \textit{far-flung Gorenstein rings}; more precisely, these are one dimensional local domains where the trace ideal of the canonical module is the conductor $\cC$. Although they defined it for only \textit{numerical semi-group} rings, the following results hold more generally under the condition $\tr_R(\omega_{R})=\cC$. 

\begin{theorem}\label{ffGcase}
	Let $(R,\m,k)$ be a one dimensional analytically unramified non-Gorenstein local domain with infinite residue field $k$ and assume that $\overline{R}$ is a DVR. Let $\omega_R$ denote the canonical module of $R$. Then $\tr_R(\omega_R)=\cC$ if and only if $\h(\omega_R)=2\lambda(R/\cC)$. 
\end{theorem}

\begin{proof}
	Let $\omega_{R}$ denote a partial trace ideal of the canonical module and suppose it has minimal reduction $z$. From \Cref{primaryeq}, we get $\h(\omega_R)=\lambda(R/\cC)+e(z;R)-\lambda(\overline{R}/R)$. We also know that $e(z;R)=\lambda(R/zR)=\lambda(\overline{R}/z\overline{R})$ \cite[Lemma 2.4]{maitra2020partial}. 
	
	Assume that $\tr_R(\omega_{R})=\cC$. Then $\omega_R\subseteq \tr_R(\omega_R)=\cC$, and hence, $\omega_R\overline{R}\subseteq \cC\overline{R}=\cC$. Since $\omega_R$ realizes the canonical module, we have $\cC=\tr_R(\omega_R)=(R:_Q\omega_R)\omega_R\subseteq\omega_R\overline{R}$ by \Cref{mainpropo}. Hence, we get that $\cC=\omega_R\overline{R}$. Now recall that $\omega_R\overline{R}=z\overline{R}$ and hence combining the above observations together, we get $\h(\omega_R)=\lambda(R/\cC)+\lambda(\overline{R}/\cC)-\lambda(\overline{R}/R)=2\lambda(R/\cC)$. This finishes one direction of the proof.
	
	For the converse, we assume that $\h(\omega_R)=2\lambda(R/\cC)$. By \Cref{primaryeq}, we immediately obtain that $e(z;R)=\lambda(\overline{R}/z\overline{R})=\lambda(\overline{R}/\cC)$. Now observe the following inclusions: $\cC\subseteq \tr_R(\omega_R)\subseteq \omega_R\overline{R}=z\overline{R}\subsetneq \overline{R}$ (see \cite[Corollary 3.6]{dao2021reflexive} for the first inclusion). Hence, $\lambda(\overline{R}/\cC)=\lambda(\overline{R}/z\overline{R})+\lambda(z\overline{R}/\tr_R(\omega_R))+\lambda(\tr_R(\omega_R)/\cC)$. Since $\lambda(\overline{R}/\cC)=\lambda(\overline{R}/z\overline{R})$, we conclude that $\cC=\tr_R(\omega_R)$. 
\end{proof}

\begin{corollary}
	Let $(R,\m,k)$ be a one dimensional analytically unramified non-Gorenstein local domain with infinite residue field $k$ and let $\omega_R$ denote the canonical module of $R$. Assume that $\overline{R}$ is a DVR  and  that $\tr_R(\omega_R)=\cC$. Then \begin{align*}
		\h(\omega_R^n)=\begin{cases}
		2\lambda(R/\cC), & n=1\\
		\lambda(R/\cC) & n\geq 2.
		\end{cases}
	\end{align*} 
\end{corollary}

\begin{proof}
	The case $n=1$ follows from \Cref{ffGcase}. Now fix $\omega_R$ to be a partial trace ideal of the canonical module. By assumption, $\omega_R\subseteq \tr_R(\omega_R)= \cC$ and hence $\omega_R^2\subseteq \cC\omega_R$. Moreover, we have $\cC\omega_R\subseteq \omega_R^2$ from \Cref{mainpropo}. Thus, we get that $\cC\omega_R=\omega_R^2$. Finally using $\omega_R$-Ulrich property of $\cC$ \cite[Corollary 4.10]{dao2021reflexive}, we get that $\omega_R^n\cong \cC$ for all $n\geq 2$. Now use \Cref{condpower}.
\end{proof}

\section{Conclusion}
We conclude this article with the following observation and question.
Note that if $I$ is isomorphic to a trace ideal $J$, then $J=\tr_R(I)$ is the  unique partial trace ideal of $I$. To see this, suppose $I\cong J=\tr_R(J)$. So, $\tr_R(I)=J$. Also, $\h(I)=\lambda(R/J)$ by \cite[Proposition 3.4]{maitra2020partial}. Now  let $I'$ be any partial trace of $I$. Then $I'\cong J$ and hence $\tr_R(I')=J$. Also, by assumption, $\h(I)=\lambda(R/I')=\lambda(R/J)=\lambda(R/\tr_R(I'))$. Thus, $I'=\tr_R(I')=J$.  This leads us to the following natural question. 

\begin{question}
	For an ideal $I$ which is not isomorphic to a trace ideal, how many partial trace ideals can exist? 
\end{question}

It is also not clear to the author if there exists a general algorithm to obtain a partial trace ideal given any ideal $I$ and as such coming up with examples is not very straightforward. Given an isomorphic copy of $I$, \Cref{mainpropo}(c) is the best possible verification procedure so far (see \Cref{strict} for an illustration).

\end{document}